\documentclass{amsart}
\usepackage{amsthm, amssymb, amsfonts, amscd}
\usepackage{graphicx}
\usepackage{stmaryrd}
\usepackage{bbold}
\usepackage{verbatim}
\usepackage[all,arc]{xy}
\usepackage{enumerate}
\usepackage{mathrsfs}
\usepackage{xcolor}
\usepackage{hyperref}
\usepackage{soul}
\usepackage{comment}
\usepackage{esint}
\usepackage[margin=1.0in]{geometry}
\usepackage{graphics}
\usepackage{enumitem}
\usepackage{tikz-cd}
\usepackage[T1]{fontenc}
\usepackage{mathtools}
\usepackage{setspace}
\allowdisplaybreaks
\DeclareMathOperator{\im}{im}
\DeclareMathOperator{\Spec}{Spec}

\DeclareMathOperator{\Core}{Core}
\DeclareMathOperator{\length}{length}
\DeclareMathOperator{\covol}{covol}

\newcommand{\mr}{\mathrm}
\newcommand{\Ind}{{\mr{Ind}}}
\newcommand{\Res}{{\mr{Res}}}
\newcommand{\Gal}{\mr{Gal}}
\newcommand{\Hom}{\mr{Hom}}

\newcommand{\GL}{\mr{GL}}

\newcommand{\tor}{\mr{tor}}

\newcommand{\bb}{\mathbb}
\newcommand{\bZ}{\bb{Z}}
\newcommand{\bQ}{\bb{Q}}
\newcommand{\bR}{\bb{R}}
\newcommand{\bC}{\bb{C}}
\newcommand{\G}{\bb{G}}
\newcommand{\bA}{\bb{A}}
\newcommand{\cO}{\mathcal{O}}
\newcommand{\trivrep}{\mathbf{1}}
\newcommand{\ol}{\overline}

\theoremstyle{definition}
\newtheorem{theorem}{Theorem}[section]

\newtheorem{proposition}[theorem]{Proposition}
\newtheorem{lemma}[theorem]{Lemma}

\begin{document}

\title[An upper bound for counting algebraic tori]{An upper bound for counting algebraic tori over $\bQ$ by Artin conductor}
\author{Jungin Lee}
\date{}
\address{J. Lee -- Department of Mathematics, Ajou University, Suwon 16499, Republic of Korea}
\email{jileemath@ajou.ac.kr}

\begin{abstract}
Let $N_n^{\tor}(X)$ be the number of isomorphism classes of $n$-dimensional algebraic tori over $\bQ$ whose Artin conductor is bounded by $X$. We prove that there is an absolute constant $C>0$ such that, for every positive integer $n \ge 2$, $N_n^{\tor}(X)\ll_n X^{\exp(C(\log n)^2)}$. The proofs of the main results were developed through an iterative dialogue with ChatGPT 5.6 Pro.
\end{abstract}
\maketitle

\vspace{-7mm}
\section{Introduction} \label{Sec1}

\subsection{Counting algebraic tori over \texorpdfstring{$\bQ$}{Q} by Artin conductor} \label{Sub11}

Throughout this paper, every number field is taken to be a subfield of a fixed algebraic closure $\ol{\bQ}$ of $\bQ$. All asymptotic statements are taken as $X\to\infty$. If nonnegative functions $f(X)$ and $g(X)$ also depend on fixed parameters $x_1,\ldots,x_m$, then $f(X)\ll_{x_1,\ldots,x_m}g(X)$ means that, after $x_1,\ldots,x_m$ are fixed, there is a positive constant $C=C(x_1,\ldots,x_m)$ such that $f(X) \le C g(X)$ for all sufficiently large $X$. 
For a number field $K$, let $D_K$ be the absolute value of its discriminant. For an integer $n \ge 2$, denote by $N_n(X)$ the number of isomorphism classes of degree-$n$ number fields $K$ satisfying $D_K \le X$. 

Counting number fields by discriminant is a central problem in arithmetic statistics. A folklore conjecture predicts that, for every $n\ge2$, there is a constant $c_n>0$ such that $N_n(X)\sim c_nX$. This conjecture is known for $n\le5$: the case $n=2$ is elementary, the case $n=3$ is due to Davenport--Heilbronn \cite{DH71} and the cases $n=4, 5$ are due to Bhargava \cite{Bha05, Bha10}. For general $n$, upper bounds have been obtained by Schmidt \cite{Sch95}, Ellenberg--Venkatesh \cite{EV06}, Couveignes \cite{Cou20}, and Lemke Oliver--Thorne \cite{LOT22}. The best known upper bound for general $n$ \cite[Theorem 1.1]{LOT22} is $N_n(X)\ll_n X^{c(\log n)^2}$ for an explicit absolute constant $c>0$.

We now consider the analogous counting problem for algebraic tori over $\bQ$, with the Artin conductor playing the role of the discriminant. Let $T$ be an $n$-dimensional torus over $\bQ$. Its character group
$$
X^*(T):=\Hom_{\ol{\bQ}}(T_{\ol{\bQ}},\G_{m,\ol{\bQ}})
$$
is a free $\bZ$-module of rank $n$ with a continuous action of $G_{\bQ}:=\Gal(\ol{\bQ}/\bQ)$ that gives a representation 
$$
\rho_T: G_{\bQ} \to \GL(X^*(T)) \cong \GL_n(\bZ)
$$ 
with a finite image. We note that the functor $T \mapsto X^*(T)$ is an anti-equivalence between the category of algebraic tori over $\bQ$ and the category of $G_{\bQ}$-lattices.
If $L$ is the splitting field of $T$, then $\ker(\rho_T)=G_L:=\Gal(\ol{\bQ}/L)$ and $G_T:=\im(\rho_T)\cong\Gal(L/\bQ)$. We write $C(T)$ for the Artin conductor of the rational representation $\rho_{T, \bQ} : G_{\bQ} \to \GL (X^*(T)_{\bQ}) \cong \GL_n(\bQ)$. Equivalently, $C(T)$ is the Artin conductor of the induced representation $\Gal(L/\bQ) \to \GL_n(\bQ)$. We refer to Neukirch \cite[Section VII.11]{Neu99} for the basic properties of Artin conductors.

For a number field $K$ of degree $n$, the Weil restriction $R_{K/\bQ}\G_m$ is an $n$-dimensional torus over $\bQ$ and satisfies $C(R_{K/\bQ}\G_m)=D_K$. (See \cite[Section 1.2]{Lee26} for more details.) Thus counting $n$-dimensional tori over $\bQ$ by Artin conductor is a natural generalization of counting degree-$n$ number fields by discriminant. Let $N_n^{\tor}(X)$ denote the number of isomorphism classes of $n$-dimensional tori over $\bQ$ such that $C(T) \le X$. The author \cite[Conjecture 3.1]{Lee26} conjectured that, for every $n \ge 1$, there is a constant $c_n>0$ such that
\begin{equation} \label{eq1a}
N_n^{\tor}(X)\sim c_nX(\log X)^{n-1}.
\end{equation}

For a finite subgroup $H\ne1$ of $\GL_n(\bZ)$, let $N_n^{\tor}(X;H)$ be the number of tori counted by $N_n^{\tor}(X)$ for which $G_T$ is conjugate to $H$ in $\GL_n(\bZ)$. We have
\begin{equation} \label{eq1a2}
N_n^{\tor}(X)=\sum_H N_n^{\tor}(X;H)+\mathbf{1}_{X \ge 1},
\end{equation}
where $H$ runs through a set of representatives for the conjugacy classes of finite nontrivial subgroups of $\GL_n(\bZ)$. The term $\mathbf{1}_{X \ge 1}$ accounts for the split torus $\G^n_m$. This is a finite sum because $\GL_n(\bZ)$ has only finitely many conjugacy classes of finite subgroups \cite[Corollary 4.8]{Bor19}. 
There is also a conjecture on the asymptotics of $N_n^{\tor}(X;H)$, which is an analogue of Malle's conjecture \cite[Conjecture 3.3]{Lee26}. We remark that this follows from a more general conjecture of Ellenberg--Venkatesh \cite[Question 4.3 and Example 4.4]{EV05}.

The conjecture \eqref{eq1a} is elementary for $n=1$. Indeed, every $1$-dimensional torus over $\bQ$ is either $\G_m$ or the norm one torus $R^{(1)}_{K/\bQ} \G_m := \ker(\mr{Nm} : R_{K/\bQ}\G_m \to \G_m)$ for a quadratic field $K$ by \cite[Example 6 in Section 4.9]{Vos98}. Since $C(R^{(1)}_{K/\bQ} \G_m) = D_K$, we have
\begin{equation*}
N_1^{\tor}(X) = N_2(X) + \mathbf{1}_{X \ge 1} = \frac{6}{\pi^2}X + O(X^{\frac{1}{2}}).
\end{equation*}

The author \cite{Lee23, Lee26} studied asymptotic upper bounds for $N_2^{\tor}(X)$ and $N_3^{\tor}(X)$. The unconditional estimates are
$$
N_2^{\tor}(X) \ll X^{1+\frac{\log 2+o(1)}{\log\log X}} \quad \text{and} \quad N_3^{\tor}(X)\ll X^{1+\frac{\log 2+o(1)}{\log\log X}}.
$$
Under the Cohen--Lenstra heuristics for $p=3$ \cite[(C6) and (C10)]{CL84}, these improve to
$$
N_2^{\tor}(X)\ll X(\log X)^{1+o(1)} \quad\text{and}\quad N_3^{\tor}(X)\ll X(\log X)^4\log\log X.
$$

Those proofs rely on explicit classifications of the finite subgroups of $\GL_2(\bZ)$ and $\GL_3(\bZ)$, together with conductor formulas for the corresponding torus types. However, such a case-by-case analysis becomes too complicated in higher dimensions. The aim of this paper is to obtain an asymptotic upper bound for $N_n^{\tor}(X)$ for arbitrary $n$ without relying on a classification of $n$-dimensional tori over $\bQ$ according to the finite subgroups of $\GL_n(\bZ)$.

\subsection{Main results} \label{Sub12}

The following theorem is the main result of the paper.

\begin{theorem} \label{thm1a}
There is an absolute constant $C>0$ such that, for every positive integer $n \ge 2$,
$$
N_n^{\tor}(X) \ll_n X^{\exp(C(\log n)^2)}.
$$
\end{theorem}

Although the exponent $\exp(C(\log n)^2)$ in Theorem \ref{thm1a} is substantially larger than the exponent $c(\log n)^2$ in the best known upper bound for $N_n(X)$, it is still subexponential in $n$. The main representation-theoretic input in the proof of Theorem \ref{thm1a} is the following theorem.

\begin{theorem} \label{thm1b}
There is an absolute constant $C>0$ such that, for every finite group $G$ and every faithful irreducible rational $G$-representation $U$ of dimension $d$, there exists $0\ne\ell\in U^\vee$ satisfying
$$
|G\cdot\ell|\le\exp(C(\log 2d)^2).
$$
\end{theorem}

The estimate for each fixed image group connects the above result to Theorem \ref{thm1a}. Let $H$ be a finite subgroup of $\GL_n(\bZ)$ and $V=\bQ^n$ be its natural rational representation. Write
$$
V \cong \trivrep^{\oplus a_0}\oplus\bigoplus_{i=1}^r W_i^{\oplus a_i},
$$
where $W_1,\ldots,W_r$ are pairwise nonisomorphic nontrivial irreducible rational $H$-representations. Set $K_i:=\ker(H\to\GL(W_i))$ and $H_i:=H/K_i$, so the induced action of $H_i$ on $W_i$ is faithful and irreducible. Write $W_i^\vee:=\Hom_{\bQ}(W_i,\bQ)$ for the dual representation. The action of $H_i$ on $W_i^\vee$ is given by
$$
(\ol{h}\cdot\ell)(w):=\ell(\ol{h}^{-1}\cdot w)
$$
for $\ol{h}\in H_i$, $\ell \in W_i^\vee$ and $w \in W_i$, where $\ol{h}^{-1}\cdot w$ is given by the action of $H_i$ on $W_i$.
Define
$$
\mu_i:=\min_{0\ne\ell\in W_i^\vee}|H_i\cdot\ell|,
$$
where $H_i\cdot\ell:=\{\ol{h}\cdot\ell :\ol{h}\in H_i\}$. Let
\begin{equation} \label{eq1c}
\Theta(H,V):=\max_{1\le i\le r}\frac{\mu_i}{2a_i}
\quad \text{and}
\quad s(H,V):=\#\left\{i:\frac{\mu_i}{2a_i}=\Theta(H,V)\right\}.
\end{equation}

\begin{theorem} \label{thm1c}
For every finite nontrivial subgroup $H$ of $\GL_n(\bZ)$,
\begin{equation*}
N_n^{\tor}(X;H)\ll_H X^{\Theta(H,V)}(\log X)^{s(H,V)-1}.
\end{equation*}
\end{theorem}

The proof of Theorem \ref{thm1c} is similar to Dummit's global tuning-lattice approach to counting Galois extensions via invariant polynomials \cite[Chapter 3]{Dum14}. 
We supplement this approach by comparing the tuning discriminant with the Artin conductor, recovering the splitting field from the normal closure of a nonzero tuning vector, and constructing low-degree invariant maps for the irreducible factors. Theorem \ref{thm1b} is proved separately using structural results on finite linear groups, and Theorem \ref{thm1a} then follows by combining Theorems \ref{thm1b} and \ref{thm1c}.

\section{Preliminaries} \label{Sec2}

We use the following notation throughout the paper. For a finite-image
rational representation $U$ of $G_{\bQ}$, let $a_p(U)$ denote the exponent of $p$ in its Artin conductor. Given a finite group $G$, a finite-dimensional rational $G$-representation $W$ and a continuous homomorphism $\varphi:G_{\bQ}\to G$, denote by $\varphi^*W$ the rational $G_{\bQ}$-representation obtained by pullback along $\varphi$, and let
$$
a_{p,W}(\varphi):=a_p(\varphi^*W)
\quad\text{and}\quad
C_W(\varphi):=\prod_p p^{a_{p,W}(\varphi)}.
$$
For $X>0$, write
$$
N(X;G,W):=\#\left\{\varphi:G_{\bQ}\twoheadrightarrow G:C_W(\varphi)\le X\right\}.
$$

Now let $H$ be a finite group, $M$ be a lattice of rank $d$ and $\rho:H\hookrightarrow\GL(M)$ be a faithful representation. Then $U:=M\otimes_{\bZ}\bQ$ can be viewed as a rational $H$-representation.
By an \emph{$H$-Galois extension}, we mean a finite Galois extension
$L/\bQ$ together with a fixed isomorphism $\Gal(L/\bQ)\cong H$. Let
$$
\varphi_L: G_{\bQ}\twoheadrightarrow\Gal(L/\bQ)\xrightarrow{\sim}H
$$
be the associated quotient map. We write
$$
a_{p,U}(L):=a_{p,U}(\varphi_L) \quad\text{and}\quad
C_U(L):=C_U(\varphi_L)=\prod_p p^{a_{p,U}(L)}.
$$
Conversely, for a continuous surjection $\varphi:G_{\bQ}\twoheadrightarrow H$, let $L_\varphi$ be the fixed field of $\ker(\varphi)$ equipped with an isomorphism $\Gal(L_{\varphi}/\bQ)\cong H$ induced by $\varphi$. Then $a_{p,U}(\varphi)=a_{p,U}(L_\varphi)$ and $C_U(\varphi)=C_U(L_\varphi)$.

\begin{lemma} \label{lem2a}
Let $H$ be a finite subgroup of $\GL_n(\bZ)$ and $V=\bQ^n$ be its natural rational representation. Then $N_n^{\tor}(X;H)\le N(X;H,V)$.
\end{lemma}

\begin{proof}
For every torus $T$ counted by $N_n^{\tor}(X;H)$, choose a basis of $X^*(T)$ for which the image of the associated Galois representation is exactly $H$. This gives a surjection $G_{\bQ}\twoheadrightarrow H$ with the same Artin conductor. If two tori give the same surjection after these choices, then their character lattices are isomorphic as $G_{\bQ}$-lattices, so the tori are isomorphic. 
\end{proof}

\subsection{Tuning submodules and small vectors} \label{Sub22}

Fix an $H$-Galois extension $L/\bQ$. Then $H$ acts on $\cO_L$ by the Galois action and acts on $M$ by $\rho$. For
$$
x=\sum_j a_j\otimes m_j\in\cO_L\otimes_{\bZ}M,
$$
we write
$$
g(x):=(g\otimes 1)x=\sum_j g(a_j)\otimes m_j
$$
and
$$
\rho(g)x:=(1\otimes\rho(g))x=\sum_j a_j\otimes\rho(g)m_j.
$$
We use the same notation on $L\otimes_{\bQ}U$. These two actions commute, so
$$
g \cdot_\rho x:=(1\otimes\rho(g))(g\otimes 1)x=\rho(g)g(x)
$$
defines a semilinear left action of $H$ on $\cO_L\otimes_{\bZ}M$. Define the \emph{tuning submodule} to be the fixed-point submodule
\begin{equation} \label{eq2a}
\Xi_\rho(L) 
:=(\cO_L\otimes_{\bZ}M)^{H, \cdot_\rho}
=\{x\in\cO_L\otimes_{\bZ}M:g(x)=\rho(g^{-1})x\text{ for every }g\in H\}.
\end{equation}
A similar construction was given by Dummit \cite[Definition 3.1]{Dum14}. The next lemma is also an analogue of \cite[Lemma 3.2]{Dum14}.

\begin{lemma} \label{lem2b}
The module $\Xi_\rho(L)$ is a free $\bZ$-module of rank $d$. Moreover, if $\xi_1,\ldots,\xi_d$ is a $\bZ$-basis of $\Xi_\rho(L)$, then it is an $L$-basis of $L\otimes_{\bQ}U$.
\end{lemma}

\begin{proof}
Let $E=L\otimes_{\bQ}U$ and $A=\cO_L\otimes_{\bZ}M$. By Galois descent for vector spaces \cite[Lemma 1.3.10]{Poo17}, $L\otimes_{\bQ}E^{H,\cdot_\rho}\cong E$, so $E^{H,\cdot_\rho}$ has dimension $d$ over $\bQ$. By \eqref{eq2a}, $\Xi_\rho(L)$ is the kernel of the $\bZ$-linear map
\begin{equation} \label{eq2b}
f : A \to \bigoplus_{g\in H}A
\quad (x \mapsto (g(x)-\rho(g^{-1})x)_{g\in H}).
\end{equation}
$A$ is a free $\bZ$-module, so its submodule $\Xi_\rho(L)$ is also free. Since $\bQ$ is flat over $\bZ$, tensoring the map $f$ with $\bQ$ gives $\Xi_\rho(L)\otimes_{\bZ}\bQ \cong E^{H,\cdot_\rho}$. Thus $\Xi_\rho(L)$ has rank $d$. 
If $\xi_1,\ldots,\xi_d$ is a $\bZ$-basis of $\Xi_\rho(L)$, then it is a $\bQ$-basis of $E^{H,\cdot_\rho}$ so $1\otimes\xi_1,\ldots,1\otimes\xi_d$ is an $L$-basis of $L\otimes_{\bQ}E^{H,\cdot_\rho}$. Under the descent isomorphism above, these vectors map to $\xi_1,\ldots,\xi_d$, respectively. We conclude that $\xi_1,\ldots,\xi_d$ is an $L$-basis of $E$.
\end{proof}

Choose a basis of $M$ and identify $\cO_L\otimes_{\bZ} M$ with $\cO_L^d$. Let $\xi_1,\ldots,\xi_d$ be a $\bZ$-basis of $\Xi_\rho(L)$, and let $B_L$ be the $d \times d$ matrix over $\cO_L$ whose $i$-th column is $\xi_i$. Then $\delta_\rho(L):=\det(B_L)$ is nonzero by Lemma \ref{lem2b}.

\begin{lemma} \label{lem2c}
The quantity $\Delta_\rho(L):=|\delta_\rho(L)^2|$ is a well-defined positive integer which is independent of the choices of the bases of $M$ and $\Xi_\rho(L)$. 
\end{lemma}

\begin{proof}
For $g\in H$, $g(B_L)=\rho(g^{-1})B_L$ gives $g(\delta_\rho(L))=\det(\rho(g^{-1}))\delta_\rho(L)$.
Since $\rho(g^{-1})\in\GL_d(\bZ)$ satisfies $\det(\rho(g^{-1}))\in\{\pm1\}$, we have $g(\delta_\rho(L)^2)=\det(\rho(g^{-1}))^2\delta_\rho(L)^2=\delta_\rho(L)^2$. 
Hence $\delta_\rho(L)^2 \in \cO_L$ is fixed by every $g \in H$, so $\delta_\rho(L)^2 \in \bZ$. Changing the basis of $M$ (resp. $\Xi_\rho(L)$) multiplies $B_L$ on the left (resp. right) by a matrix in $\GL_d(\bZ)$. Since every element of $\GL_d(\bZ)$ has determinant $\pm 1$, $\Delta_\rho(L)$ is independent of the choices of the bases of $M$ and $\Xi_\rho(L)$. 
\end{proof}

For a prime $p$, write $L_p:=L\otimes_{\bQ}\bQ_p$, $\cO_{L,p}:=\cO_L\otimes_{\bZ}\bZ_p$, $M_p:=M\otimes_{\bZ}\bZ_p$ and $U_p:=M_p\otimes_{\bZ_p}\bQ_p \cong U\otimes_{\bQ}\bQ_p$. Let $\rho_p: H \to \GL(M_p)$ and $\rho_{p, \bQ_p}: H \to \GL(U_p)$ be the representations induced by $\rho$. 
Fix the isomorphism $\Gal(L/\bQ)\cong H$ and choose an embedding $\ol{\bQ}\hookrightarrow\ol{\bQ}_p$ extending $\bQ\hookrightarrow\bQ_p$, which identifies $G_{\bQ_p}$ as a subgroup of $G_{\bQ}$. Define a homomorphism
$$
\psi_p:G_{\bQ_p}\hookrightarrow G_{\bQ}\twoheadrightarrow\Gal(L/\bQ)\cong H.
$$
Denote the pointed $H$-torsor corresponding to $\psi_p$ by $T_{\psi_p}$. It is isomorphic to the base change of $\Spec L\to\Spec\bQ$ to $\Spec\bQ_p$, and hence $T_{\psi_p}\cong\Spec L_p\rightarrow\Spec\bQ_p$.

Since $\bZ_p$ is flat over $\bZ$, tensoring \eqref{eq2b} with $\bZ_p$ gives
\begin{equation} \label{eq2c}
\Xi_\rho(L)\otimes_{\bZ}\bZ_p =\{x\in\cO_{L,p}\otimes_{\bZ_p}M_p: g(x)=\rho_p(g^{-1})x\text{ for every }g\in H\}.
\end{equation}
Let $\delta_{\psi_p}$ denote the action $\delta$ of $H$ on $L_p$ appearing in Wood--Yasuda \cite[p. 12597]{WY15}. Under the pointed torsor convention used there, this action is opposite to our standard left Galois action; it is given by
$$
\delta_{\psi_p}(g)(a\otimes m)=g^{-1}(a)\otimes m,
$$
whereas the action induced by $\rho_p$ is
$$
\widetilde\rho_p(g)(a\otimes m)=a\otimes\rho_p(g)m.
$$
Consequently, the Wood--Yasuda tuning condition $\delta_{\psi_p}(g)x=\widetilde\rho_p(g)x$ is equivalent to $g(x)=\rho_p(g^{-1})x$. Hence the right-hand side of \eqref{eq2c} is the local tuning submodule of Wood--Yasuda \cite[Definition 3.1]{WY15}.

Write $\delta_L:=\delta_\rho(L)$ for simplicity. Recall that $\xi_1, \ldots, \xi_d$ is a $\bZ$-basis of $\Xi_\rho(L)$ and the corresponding matrix $B_L$ (with respect to the basis of $M$) has determinant $\delta_L$. Then $\xi_1 \otimes 1, \ldots, \xi_d \otimes 1$ is a $\bZ_p$-basis of $\Xi_\rho(L)\otimes_{\bZ}\bZ_p$ and the corresponding matrix has the same determinant $\delta_L$, viewed as an element of $\cO_{L,p}$. Define
$$
v_{\rho,p}(L):=v_{\bQ_p,H,\rho_p}(\psi_p) =\frac{1}{|H|}\length_{\bZ_p} \left(\cO_{L,p}/\delta_L\cO_{L,p}\right),
$$
where $v_{\bQ_p,H,\rho_p}$ denotes the $v$-invariant of Wood--Yasuda \cite[Definition 3.3]{WY15}. Since $\cO_{L,p}$ is a free $\bZ_p$-module of rank $|H|$, we have 
\begin{equation} \label{eq2d}
v_p(\Delta_\rho(L))=v_p(\delta_L^2)=\frac{1}{|H|}\length_{\bZ_p} \left(\cO_{L,p}/\delta_L^2\cO_{L,p}\right)=2v_{\rho,p}(L).
\end{equation}

\begin{proposition} \label{prop2e}
There is a constant $c_{H,\rho}>0$ such that every $H$-Galois extension $L/\bQ$ satisfies
\begin{equation*}
\Delta_\rho(L)\le c_{H,\rho} C_U(L).
\end{equation*}
\end{proposition}

\begin{proof}
Suppose first that $p\nmid |H|$. The image under $\psi_p$ of the wild inertia subgroup of $G_{\bQ_p}$ is a $p$-subgroup of $H$, so it is trivial. Thus $\psi_p$ is tamely ramified, and the local data $(\bQ_p,H,\rho_p,\psi_p)$ satisfy the hypotheses of the tame setting in \cite[Section 4.1]{WY15}. Since $\rho$ is defined over $\bQ\subset\bR$, the eigenvalues of every element of $H$ occur in inverse pairs away from $\pm1$, so $\rho_p$ is \emph{balanced} in the sense of Wood--Yasuda \cite[Definition 4.4]{WY15}. Their tame comparison \cite[Proposition 4.5]{WY15} gives
$$
v_{\bQ_p,H,\rho_p}(\psi_p)=\frac{1}{2}a_{\bQ_p,H,\rho_p}(\psi_p),
$$
where $a_{\bQ_p,H,\rho_p}(\psi_p)$ denotes the Artin conductor of the local representation
$$
G_{\bQ_p}\xrightarrow{\psi_p}H\xrightarrow{\rho_{p, \bQ_p}}\GL(U_p).
$$
Since $\psi_p=\varphi_L|_{G_{\bQ_p}}$, this representation is
naturally isomorphic to $(\varphi_L^*U)|_{G_{\bQ_p}}\otimes_{\bQ}\bQ_p$, so
$$
a_{\bQ_p,H,\rho_p}(\psi_p)=a_{p,U}(L).
$$
Thus \eqref{eq2d} implies that $v_p(\Delta_\rho(L))=a_{p,U}(L)$.

Now suppose that $p\mid |H|$. There are only finitely many continuous homomorphisms from $G_{\bQ_p}$ to $H$, since $\bQ_p$ has only finitely many extensions of bounded degree. Hence
$$
b_{p,H,\rho} :=\max_{\psi \in \Hom_{\mr{cont}}(G_{\bQ_p}, H)} 2v_{\bQ_p,H,\rho_p}(\psi)
$$
is finite. For the homomorphism $\psi_p$ arising from $L$, \eqref{eq2d} gives 
$$
v_p(\Delta_\rho(L)) = 2v_{\bQ_p,H,\rho_p}(\psi_p) \le b_{p,H,\rho}.
$$
We conclude that
\begin{equation*}
\Delta_\rho(L) =\prod_{p\nmid |H|}p^{a_{p,U}(L)} \prod_{p\mid |H|}p^{v_p(\Delta_\rho(L))}
\le \left(\prod_{p\mid |H|}p^{b_{p,H,\rho}}\right) C_U(L). \qedhere
\end{equation*}
\end{proof}

Fix an $H$-invariant positive definite inner product on $U_{\bR}:=U\otimes_{\bQ}\bR$. Extend it to the associated Hermitian inner product on $U\otimes_{\bQ}\bC$, and denote the corresponding norm by $\|\cdot\|$. Choose an embedding $\iota:L\hookrightarrow\bC$. For the complex conjugation $\tau$, $\tau \circ \iota : L \to \bC$ is also an embedding. Since $L/\bQ$ is Galois, there is a unique element $c = c_{\iota} \in H$ such that $\iota(ca)=(\tau \circ \iota)(a)=\ol{\iota(a)}$ for every $a \in L$. The relation $\iota(c^2a)=\ol{\iota(ca)}=\iota(a)$ gives $c^2=1$.
For an element $x=\sum_j a_j\otimes m_j\in\Xi_\rho(L)$, we have
$$
\ol{\iota(x)} =\sum_j\ol{\iota(a_j)}\otimes m_j =\sum_j\iota(ca_j)\otimes m_j =\iota(c(x)).
$$
By the tuning relation, $\iota(c(x))=\rho(c^{-1})\iota(x)=\rho(c)\iota(x)$, so $\iota(\Xi_\rho(L))$ lies in the real vector space
\begin{equation*}
W_c:=\{z\in U\otimes_{\bQ}\bC:\ol z=\rho(c)z\}.
\end{equation*}

If $U_{\bR}=U_+\oplus U_-$ is the decomposition into the $\pm1$ eigenspaces of $\rho(c)$, then $W_c=U_+\oplus iU_-$ and $W_c$ has real dimension $d$. 
We equip $W_c$ with the restriction of the Hermitian norm $\|\cdot\|$ on $U\otimes_{\bQ}\bC$; explicitly,
$$
\|u+iv\|^2=\|u\|^2+\|v\|^2
$$
for $u\in U_+$ and $v\in U_-$.

\begin{lemma} \label{lem2g}
There is a constant $b_{H,\rho}>0$ such that
\begin{equation*}
\covol(\iota(\Xi_\rho(L)))\le b_{H,\rho}\Delta_\rho(L)^{1/2}.
\end{equation*}
\end{lemma}

\begin{proof}
For each $c\in H$ satisfying $c^2=1$, choose an orthonormal real basis $e_{c,1},\ldots,e_{c,d}$ of $W_c$. Since $W_c$ is a real form of $U\otimes_{\bQ}\bC$, this is also a complex basis of $U\otimes_{\bQ}\bC$. Let $P_c\in\GL_d(\bC)$ be the matrix whose columns are the coordinates of $e_{c,1},\ldots,e_{c,d}$ with respect to the fixed basis of $M$. 
Let $A_L\in \mr{M}_d(\bR)$ be the matrix whose columns are the coordinates of $\iota(\xi_1),\ldots,\iota(\xi_d)$ with respect to
$e_{c,1},\ldots,e_{c,d}$. Then $\iota(B_L)=P_cA_L$.

Since the chosen basis of $W_c$ is orthonormal, we have
$$
\covol\bigl(\iota(\Xi_\rho(L))\bigr)
=|\det(A_L)|
=|\det(P_c)|^{-1}|\det(\iota(B_L))|
=|\det(P_c)|^{-1}|\iota(\delta_\rho(L))|.
$$
Moreover, $\delta_\rho(L)^2\in\bZ$ gives $|\iota(\delta_\rho(L))|=\Delta_\rho(L)^{1/2}$. 
Finally, there are only finitely many choices of $c \in H$ such that $c^2=1$, so the quantities $|\det(P_c)|^{-1}$ are bounded
uniformly in $c$.
\end{proof}

\begin{proposition} \label{prop2h}
There is a constant $A_{H,\rho}>0$ such that every $H$-Galois extension $L/\bQ$ admits a nonzero vector $x\in\Xi_\rho(L)$ such that
\begin{equation} \label{eq2e}
\|\iota(x)\| \le A_{H,\rho}C_U(L)^{1/(2d)}.
\end{equation}
\end{proposition}

\begin{proof}
By Proposition \ref{prop2e} and Lemma \ref{lem2g}, we have $\covol(\iota(\Xi_\rho(L)))\ll_{H,\rho}C_U(L)^{1/2}$. The assertion follows from Minkowski's first theorem \cite[Chapter III, Section 2]{Cas59}.
\end{proof}

\subsection{Finite invariant maps} \label{Sub23}

In this subsection, we assume that $U$ is irreducible over $\bQ$. 

\begin{lemma} \label{lem2i}
Let $x\in\Xi_\rho(L)$ be a nonzero vector, and choose a basis of $M$ so that $x=(x_1,\ldots,x_d)$. Then the normal closure of $E_x:=\bQ(x_1,\ldots,x_d)$ over $\bQ$ is $L$.
\end{lemma}

\begin{proof}
Let $S_x:=\{h\in H:h(x_j)=x_j\text{ for every }j\}$. Then $S_x=\Gal(L/E_x)$, so $E_x=L^{S_x}$. Let
$$
N:=\Core_H(S_x)=\bigcap_{h\in H}hS_xh^{-1}.
$$
For $n\in N$, $n^{-1} \in N \subseteq S_x$ so the
tuning relation gives $\rho(n)x=n^{-1}(x)=x$. Now let $N$ act on $L\otimes_{\bQ}U$ trivially on the first factor
and through $\rho$ on the second factor. Then 
$$
0 \ne x\in(L\otimes_{\bQ}U)^N=L\otimes_{\bQ}U^N,
$$
and therefore $U^N\ne 0$. Since $N$ is normal in $H$, the subspace
$U^N$ is an $H$-subrepresentation of the irreducible
$H$-representation $U$. Hence $U^N=U$ and $\rho$ is faithful, so $N=1$. Now the normal closure of $L^{S_x}$ in $L$ is $L^{\Core_H(S_x)}=L$, as required.
\end{proof}

Write $M^\vee:=\Hom_{\bZ}(M,\bZ)$ for the dual lattice and $\bZ[M]:=\operatorname{Sym}_{\bZ}(M^\vee)$ for the ring of integral polynomial functions on $M$. Suppose that the map
\begin{equation*}
F=(f_1,\ldots,f_d):U\rightarrow\bA^d
\end{equation*}
is a finite $H$-invariant morphism defined over $\bQ$. After multiplying each $f_i$ by a suitable nonzero integer, we may assume that $f_i\in\bZ[M]^H$ for each $i$. 

\begin{proposition} \label{prop2j}
With the notation above,
\begin{equation} \label{eq2f}
N(X;H,U)\ll_{H,\rho,F}X^{\frac{1}{2d}\sum_{i=1}^{d} \deg f_i}.
\end{equation}
\end{proposition}

\begin{proof}
Write $k_i:=\deg f_i$. For every surjection $\varphi:G_{\bQ}\twoheadrightarrow H$ counted by $N(X;H,U)$, let $L=L_\varphi$ and choose an embedding $\iota:L\hookrightarrow\bC$. By Proposition \ref{prop2h} and the inequality $C_U(L)\le X$, there exists a nonzero
$$
x_\varphi=(x_{\varphi,1},\ldots,x_{\varphi,d})\in\Xi_\rho(L)
$$ 
such that $\|\iota(x_\varphi)\| \ll_{H,\rho} X^{1/(2d)}$. For each $\varphi$, fix one such vector $x_\varphi$. Since $f_i$ is $H$-invariant, we have
$$
h(f_i(x_\varphi))=f_i(h(x_\varphi))
=f_i(\rho(h^{-1})x_\varphi)=f_i(x_\varphi)
$$
for every $h\in H$, so $f_i(x_\varphi)\in\bQ$. Moreover, since $x_{\varphi,1},\ldots,x_{\varphi,d}\in\cO_L$ and $f_i$ has coefficients in $\bZ$, we have $f_i(x_\varphi)\in\bZ$. It is clear that $|f_i(x_\varphi)|=|\iota(f_i(x_\varphi))|=|f_i(\iota(x_\varphi))|$.

As $f_i$ is a fixed polynomial of degree $k_i$, $|f_i(z)|\ll_{U,F}(1+\|z\|)^{k_i}$ for $z\in U\otimes_{\bQ}\bC$. Since $U=M\otimes_{\bZ}\bQ$ is fixed once $\rho:H\to\GL(M)$ is fixed, we absorb the dependence on $U$ into that on $\rho$. Combining this with the bound $\|\iota(x_\varphi)\| \ll_{H,\rho}X^{1/(2d)}$, we obtain
$$
|f_i(x_\varphi)|=|f_i(\iota(x_\varphi))|\ll_{H,\rho,F} X^{k_i/(2d)}.
$$
Hence the number of possible tuples $F(x_\varphi)=(f_1(x_\varphi),\ldots,f_d(x_\varphi)) \in\bZ^d$ is
$$
\ll_{H,\rho,F}\prod_{i=1}^dX^{k_i/(2d)}=X^{(k_1+\cdots+k_d)/(2d)}.
$$
Since $F$ is a finite morphism, the number of possible $x_\varphi \in \Xi_\rho(L)$ is also $\ll_{H,\rho,F} X^{(k_1+\cdots+k_d)/(2d)}$.

By Lemma \ref{lem2i}, each such $x_\varphi$ determines $L$ as the normal closure of $\bQ(x_{\varphi,1},\ldots,x_{\varphi,d})/\bQ$. For a fixed $L$, the corresponding surjections $\varphi$ are parametrized by isomorphisms $\Gal(L/\bQ) \cong H$, and there are $O_H(1)$ such isomorphisms. We conclude that \eqref{eq2f} holds.
\end{proof}

We next construct a finite invariant map from one orbit of covectors. Define
\begin{equation*}
\mu(H,U):=\min_{0\ne\ell\in U^\vee}|H\cdot\ell|.
\end{equation*}

\begin{lemma} \label{lem2k}
Let $m=\mu(H,U)$. There exists a finite $H$-invariant morphism $F:U\to\bA^d$ defined by integral polynomials of degree at most $m$.
\end{lemma}

\begin{proof}
Choose $0\ne\ell\in U^\vee$ whose $H$-orbit has size $m$. By multiplying $\ell$ by a suitable nonzero integer if necessary, we may assume that $\ell \in M^\vee$. Write $H\cdot\ell=\{\ell_1,\ldots,\ell_m\}$. Since $U^\vee$ is irreducible, the orbit $H\cdot\ell$ spans $U^\vee$. Hence $\bigcap_{i=1}^{m} \ker(\ell_i) = 0$, so the map $j:U\to\bA^m$ given by $j(v)=(\ell_1(v),\ldots,\ell_m(v))$ is a closed immersion.

Define the map $q:\bA^m\to\bA^m$ by $q(y)=(e_1(y),\ldots,e_m(y))$, where $e_i$ is the $i$-th elementary symmetric polynomial in $m$ variables. The map $q$ is finite, and $q\circ j$ is $H$-invariant because $H$ permutes $\ell_1, \ldots, \ell_m$. Moreover, the image $Y:=(q\circ j)(U)$ is a closed affine variety of dimension $d$, and $U\to Y$ is a finite morphism. 
A version of Noether’s normalization lemma \cite[Theorem 13.3]{Eis95}, together with its geometric interpretation in \cite[p. 284]{Eis95}, gives a rational linear projection $\pi:\bA^m\to\bA^d$ whose restriction to $Y$ is finite. Now $F = (f_1, \ldots, f_d) :=\pi\circ q\circ j:U\to\bA^d$ is finite and $H$-invariant. Its coordinate functions are linear combinations of $e_1,\ldots,e_m$, so their degrees are at most $m$. Clearing denominators in each $f_i$ preserves finiteness of $F$ and gives integral coefficients.
\end{proof}

Applying Proposition \ref{prop2j} to the map $F$ of Lemma \ref{lem2k} gives
\begin{equation} \label{eq2g}
N(X;H,U) \ll_{H,\rho} X^{(d\mu(H,U))/(2d)} = X^{\mu(H,U)/2}.
\end{equation}

\section{The general upper bound} \label{Sec3}

Fix a finite nontrivial subgroup $H$ of $\GL_n(\bZ)$, and retain the notation $V$, $W_i$, $a_i$, $K_i$ and $H_i$ introduced in Section \ref{Sub12}. 
Fix a surjection $\varphi:G_{\bQ}\twoheadrightarrow H$, and let
$L=L_\varphi$ be the corresponding $H$-Galois extension. The quotient map $\varphi_i:G_{\bQ}\twoheadrightarrow H_i$ corresponds to the field $L_i:=L^{K_i}$.
Since the original representation of $H$ on $V$ is faithful, $\bigcap_{i=1}^rK_i=1$. Thus the homomorphism $H\to\prod_{i=1}^rH_i$ is injective, and the tuple $(\varphi_1,\ldots,\varphi_r)$ determines $\varphi$. 

Let $C_i(\varphi_i):=C_{W_i}(\varphi_i)$. Since $\varphi^*V\cong\trivrep^{\oplus a_0}\oplus\bigoplus_{i=1}^r(\varphi_i^*W_i)^{\oplus a_i}$, we have
\begin{equation} \label{eq3b}
C_V(\varphi)=\prod_{i=1}^r C_i(\varphi_i)^{a_i}.
\end{equation}
Let $\Theta:=\Theta(H,V)$ and $s:=s(H,V)$, with the notation of \eqref{eq1c}.

\begin{proof}[Proof of Theorem \ref{thm1c}]
Since $H_i$ is finite, $W_i$ contains an $H_i$-stable lattice. Hence for every $i$, \eqref{eq2g} gives
\begin{equation} \label{eq3c}
\#\{\varphi_i:G_{\bQ}\twoheadrightarrow H_i:C_i(\varphi_i)\le Y\} 
= N(Y; H_i, W_i) \ll_H Y^{\mu(H_i, W_i)/2} = Y^{\mu_i/2}.
\end{equation}
By \eqref{eq3b}, every surjection $\varphi:G_{\bQ}\twoheadrightarrow H$ with $C_V(\varphi)\le X$ is determined by the tuple $(\varphi_1,\ldots,\varphi_r)$ satisfying $\prod_{i=1}^rC_i(\varphi_i)^{a_i}\le X$.

By permuting the indices, we may assume that $\frac{\mu_i}{2a_i}=\Theta$ for $i \le s$ and $\frac{\mu_i}{2a_i} < \Theta$ for $i>s$. Write $\beta_i:=\mu_i/2$, $\delta_i:=\Theta a_i-\beta_i$ and $L:=\lfloor\log_2X\rfloor$. Then $\delta_i=0$ for $i\le s$ and $\delta_i>0$ for $i>s$.
Using the dyadic ranges $2^{k_i}\le C_i(\varphi_i)<2^{k_i+1}$, \eqref{eq3c} gives
\begin{align*}
N(X;H,V)
&\le \sum_{\substack{k_1,\ldots,k_r\ge0\\ \sum_{i=1}^{r} a_i k_i\le L}}
\prod_{i=1}^r \#\left\{ \varphi_i:G_{\bQ}\twoheadrightarrow H_i:
2^{k_i}\le C_i(\varphi_i)<2^{k_i+1} \right\} \\
&\ll_H \sum_{\substack{k_1,\ldots,k_r\ge0\\ \sum_{i=1}^{r} a_i k_i\le L}} \prod_{i=1}^r 2^{(k_i+1)\beta_i} \\
&\ll_H \sum_{0\le m\le L}2^{\Theta m} \sum_{\substack{k_1,\ldots,k_r\ge0\\ \sum_{i=1}^{r} a_i k_i=m}} 2^{-\sum_{i>s} \delta_i k_i}.
\end{align*}
We also have
\begin{align*}
\sum_{\substack{k_1,\ldots,k_r\ge0\\ \sum_{i=1}^{r} a_i k_i=m}} 2^{-\sum_{i>s} \delta_i k_i}
&\le \sum_{k_{s+1},\ldots,k_r\ge0} 2^{-\sum_{i>s} \delta_i k_i}
\# \left\{ (k_1, \ldots, k_s) \in \bZ_{\ge 0}^{s} : \sum_{i=1}^{s} a_i k_i=m-\sum_{i=s+1}^{r} a_i k_i \right\} \\
&\ll_H (1+m)^{s-1} \sum_{k_{s+1},\ldots,k_r\ge0} 2^{-\sum_{i>s} \delta_i k_i} \\
&\ll_H (1+m)^{s-1},
\end{align*}
so
$$
N(X;H,V) \ll_H \sum_{0\le m\le L}2^{\Theta m}(1+m)^{s-1} \ll_H 2^{\Theta L}(1+L)^{s-1} \ll_H X^\Theta(\log X)^{s-1}.
$$
By Lemma \ref{lem2a}, $N_n^{\tor}(X; H) \le N(X;H,V)$. This completes the proof.
\end{proof}

Next we move to the proof of Theorem \ref{thm1b}. Before proving Theorem \ref{thm1b}, we introduce several lemmas. 
Let $m \ge 2$ be an integer, let $S_m$ (resp. $A_m$) denote the symmetric (resp. alternating) group on $m$ letters and write $\lambda \vdash m$ if $\lambda$ is a partition of $m$. For $\lambda=(\lambda_1, \ldots, \lambda_t) \vdash m$, let $\lambda'$ be its conjugate partition and $S_\lambda:=S_{\lambda_1}\times S_{\lambda_2}\times\cdots \times S_{\lambda_t}\le S_m$ be the corresponding Young subgroup. Write $S^\lambda$ for the complex Specht module associated with $\lambda$, and let $f^\lambda:=\dim_{\bC}S^\lambda$. Over $\bC$, this module affords the irreducible $S_m$-representation denoted by $[\lambda]$ in \cite{JK81}; see \cite[Theorem 3.1.10 and Lemma 7.1.4]{JK81}. In the following, we identify $S^\lambda$ with $[\lambda]$ and use the notation $S^\lambda$ throughout.

\begin{lemma} \label{lem3d}
There is an absolute constant $c_1>0$ with the following property. Let $m\ge10$ and $X$ be a nontrivial irreducible complex representation of $A_m$ of dimension $a \ge 1$. Then there exists a subgroup $H\le A_m$ such that $X^H\ne0$ and 
$$
[A_m:H]\le\exp(c_1(\log 2a)^2).
$$
\end{lemma}

\begin{proof}
Let $\nu\vdash m$ be neither $(m)$ nor $(1^m)$. Choose $\mu\in\{\nu,\nu'\}$ such that $p:=\mu_1\ge q:=\mu'_1$. Let $k:=m-p$ and write $f:=f^\mu=f^\nu$. Then
\begin{equation} \label{eq3d}
[S_m:S_\mu]=\frac{m!}{\prod_i\mu_i!}
\le\frac{m!}{p!}\le m^k.
\end{equation}

The dimension $f$ is equal to the number of standard tableaux of shape $\mu$ \cite[Corollary 7.2.8]{JK81}, which is given by the hook length formula \cite[Theorem 2.3.21]{JK81}. Since there are $\binom{p+q-2}{q-1}$ ways to fill the first row and the first column with the numbers $1, 2, \ldots, p+q-1$ so that the entries increase along rows and columns, and each such filling extends to a standard tableau of shape $\mu$, we have 
\begin{equation} \label{eq3f}
f \ge \binom{p+q-2}{q-1} \ge \binom{2q-2}{q-1} \ge 2^{q-1}.
\end{equation}
For $s\ge2$, consider the anti-diagonals $D_s:=\{(i,j)\in\mu:i+j=s\}$. Then ordering the anti-diagonals $D_2, D_3, \ldots$ successively and ordering the boxes within each $D_s$ arbitrarily gives a standard tableau. Since $r!\ge2^{r-1}$ for $r\ge1$ and there are at most $p+q-1$ nonempty anti-diagonals, we obtain
\begin{equation} \label{eq3g}
f\ge\prod_s|D_s|!
\ge2^{m-(p+q-1)}=2^{k-q+1}.
\end{equation}

By \eqref{eq3f} and \eqref{eq3g}, we have $f^2 \ge 2^{q-1}2^{k-q+1}=2^k$, so $k \ll \log 2f$. Moreover, $q \ge 2$ by the assumption that $\nu$ is neither $(m)$ nor $(1^m)$, so \eqref{eq3f} gives $f \ge \binom{p+q-2}{q-1} \ge p$. Therefore
$$
m=p+k \le f+\frac{2}{\log 2}\log f
$$
so $\log m \ll \log 2f$. Combining \eqref{eq3d} with the bounds $k\ll \log 2f$ and $\log m \ll \log 2f$, we obtain
\begin{equation} \label{eq3h}
[S_m:S_\mu]\le \exp(k\log m) \le \exp(c_0(\log 2f)^2)
\end{equation}
for an absolute constant $c_0>0$.

By the classification of irreducible complex representations of $A_m$ in \cite[Theorem 2.5.7]{JK81}, there is a partition $\lambda\vdash m$ such that either $\lambda\ne\lambda'$ and $X\cong S^\lambda|_{A_m} \cong S^{\lambda'}|_{A_m}$, or $\lambda=\lambda'$ and $X$ is isomorphic to one of the two irreducible constituents of $S^\lambda|_{A_m}$. In either case, $\lambda$ is neither $(m)$ nor $(1^m)$.
Now we apply the preceding argument with $\nu=\lambda$. Choose $\mu\in\{\lambda,\lambda'\}$ as above and set $f:=f^\mu=f^\lambda$. If $\lambda\ne\lambda'$, then $a=f$; if $\lambda=\lambda'$, then the two constituents of $S^\lambda|_{A_m}$ have dimension $f/2$, so $a=f/2$. Thus $f\le2a$.

By \cite[Theorem 2.1.3 and (2.1.4)]{JK81}, $S^\mu$ occurs with multiplicity one in $\Ind_{S_\mu}^{S_m}\trivrep$. By Frobenius reciprocity \cite[Section 7.2, Theorem 13]{Ser77}, we have
$$
\dim_{\bC}(S^\mu)^{S_\mu} = \left\langle \Ind_{S_\mu}^{S_m}\trivrep,\chi^\mu\right\rangle_{S_m} = \left\langle \trivrep, \Res_{S_\mu}^{S_m}\chi^\mu\right\rangle_{S_\mu}=1,
$$
where $\chi^\mu$ denotes the character of $S^\mu$, and $(S^\mu)^{S_\mu}$ denotes the fixed subspace of $S^\mu$ under the action of $S_\mu$.
Consequently, $S^\mu|_{A_m}$ has a nonzero vector fixed by $H_\mu:=A_m\cap S_\mu$. If $S^\mu|_{A_m}$ is irreducible, then it is isomorphic to $X$. If $S^\mu|_{A_m}$ has two irreducible constituents, then 
a nonzero vector in $S^\mu|_{A_m}$ fixed by $H_\mu$ has a nonzero projection to at least one constituent, and conjugation by an odd permutation interchanges the two constituents \cite[Theorem 2.5.7]{JK81}. Conjugating $H_\mu$ by that odd permutation if necessary, we obtain a subgroup $H\le A_m$ such that $X^H\ne0$. Since $H$ is conjugate to $H_\mu$ in $S_m$, \eqref{eq3h} (applied to $\nu=\lambda$) and the inequality $f \le 2a$ give
\begin{equation*}
[A_m:H]=[A_m:H_\mu] \le[S_m:S_\mu]
\le\exp(c_0(\log 2f)^2) \le\exp(4c_0(\log 2a)^2). \qedhere
\end{equation*}
\end{proof}

For a finite-dimensional complex vector space $W$, a finite subgroup $\Gamma\le\GL(W)$ is called \emph{primitive} if it acts irreducibly on $W$ and there is no decomposition $W=W_1\oplus\cdots\oplus W_r$ into $r \ge 2$ nonzero subspaces that are permuted transitively by $\Gamma$. 
The next lemma is a special case of \cite[Theorem 1.3]{Wei12}. In Weisfeiler’s notation, the characteristic exponent of $\bC$ is $1$ and groups of Lie $1$-type are trivial; hence the Lie-type subgroup $L$ appearing in that theorem is trivial. We note that both occurrences of $l$ in \cite[p. 1]{Wei12} are typographical errors for the numeral $1$.

\begin{lemma}[Weisfeiler] \label{lem3e}
Let $W$ be an $n$-dimensional complex vector space, 
$\Gamma\le\GL(W)$ be a finite primitive linear group and $Z=Z(\Gamma)$ be the center of $\Gamma$. Then $\Gamma$ contains a normal subgroup $N=N_1\times\cdots\times N_t$ ($N=1$ if $t=0$) such that for each $i$, $N_i \cong A_{m_i}$ for some $m_i \ge 10$, and 
$[\Gamma:ZN]\le n^{2\log_2n+5}$.
\end{lemma}

For a finite linear group $\Gamma\le\GL(W)$, define
$$
\Omega(\Gamma):= \max\{\mr{ord}(\alpha):\alpha\text{ is an eigenvalue for some }\gamma\in\Gamma\}.
$$

\begin{lemma} \label{lem3f}
There is an absolute constant $c_2>0$ such that, for every finite subgroup $\Gamma\le\GL(W)$ acting irreducibly on $W$ with $\dim_{\bC} W = n$, there exists $0\ne w\in W$ satisfying 
$$
|\Gamma\cdot w| \le\Omega(\Gamma) \exp(c_2(\log 2n)^2).
$$
\end{lemma}

\begin{proof}
Use induction on $n$. If $n=1$, then $\Gamma$ is a finite subgroup of $\bC^\times$, so every nonzero vector has orbit size $|\Gamma|=\Omega(\Gamma)$. 
Suppose that $\Gamma$ is imprimitive. Write $W=W_1\oplus\cdots\oplus W_r$ ($r \ge 2$), where the $W_i$ are permuted transitively by $\Gamma$. Let $s:=\dim W_1=n/r$, $\Gamma_1$ be the stabilizer of $W_1$ and $\ol{\Gamma}_1$ be its image in $\GL(W_1)$. The group $\ol{\Gamma}_1$ acts irreducibly on $W_1$. Indeed, if $0\ne U\subsetneq W_1$ were $\ol{\Gamma}_1$-invariant, then the sum
$\sum_{\gamma\in\Gamma}\gamma U$ would be a nonzero proper $\Gamma$-invariant subspace of $W$, contradicting the irreducibility of the $\Gamma$-action on $W$. 
Moreover, every eigenvalue of the restriction of an element of $\Gamma_1$ to $W_1$ is also an eigenvalue of that element on $W$, which gives $\Omega(\ol{\Gamma}_1)\le\Omega(\Gamma)$.

By the induction hypothesis, there exists $0\ne w\in W_1$ such that
$$
|\ol{\Gamma}_1\cdot w|
\le\Omega(\ol{\Gamma}_1)\exp(c_2(\log 2s)^2)
\le\Omega(\Gamma)\exp(c_2(\log 2s)^2).
$$
If $\gamma\in\Gamma$ fixes $w$, then $w \in \gamma(W_1) \cap W_1$. Since $\Gamma$ permutes the summands $W_1,\ldots,W_r$, we have
$\gamma(W_1)=W_1$ and hence $\gamma\in\Gamma_1$.
Thus the stabilizer of $w$ lies in $\Gamma_1$ and $|\Gamma\cdot w|=r|\ol{\Gamma}_1\cdot w|$.
Since
$$
(\log 2n)^2-(\log 2s)^2 
=\log r\bigl(\log 2n+\log 2s\bigr)
\ge (2 \log2)\log r,
$$
after taking $c_2\ge(2\log2)^{-1}$, we have 
$$
|\Gamma \cdot w| = r|\ol{\Gamma}_1\cdot w| \le r\Omega(\Gamma)\exp(c_2(\log 2s)^2) \le \Omega(\Gamma)\exp(c_2(\log 2n)^2).
$$

Hence we may and will assume that $\Gamma$ is primitive. Let $Z:=Z(\Gamma)$ and choose $N=N_1\times\cdots\times N_t$ as in Lemma \ref{lem3e}. 
By Clifford's theorem \cite[Theorem 6.2]{Isa76}, there are pairwise nonisomorphic irreducible $N$-representations $Y_1,\ldots,Y_l$ and an integer $u\ge 1$ such that $Y_1,\ldots,Y_l$ are precisely the conjugate representations $Y_1^\gamma$ ($\gamma\in\Gamma$) up to isomorphism, and
$$
W|_N\cong\bigoplus_{j=1}^l Y_j^{\oplus u}.
$$
Here, $Y_1^\gamma(n):=Y_1(\gamma n\gamma^{-1})$ for $\gamma\in\Gamma$ and $n\in N$.
If $l \ge 2$, then the components $Y_j^{\oplus u}$ ($1 \le j \le l$) are permuted transitively by $\Gamma$, which contradicts the primitivity of $\Gamma$. Thus $l=1$ and $W|_N\cong X^{\oplus u}$ for some irreducible $N$-representation $X$.

Iterating \cite[Section 3.2, Theorem 10(ii)]{Ser77}, we may write
$$
X\cong X_1\boxtimes\cdots\boxtimes X_t,
$$
where $X_i$ is an irreducible $N_i$-representation.
No $N_i$ acts trivially on $X$, since otherwise $N_i$ would act trivially on $W$. This is impossible because the inclusion $N_i\le\Gamma\le\GL(W)$ makes the action of $N_i$ on $W$
faithful. Since $N_i \cong A_{m_i}$ is a perfect group, it has no nontrivial one-dimensional representations. Thus 
\begin{equation} \label{eq3i1}
a_i:=\dim X_i \ge 2 \quad \text{and} \quad \prod_{i=1}^ta_i=\dim X\le n.
\end{equation}

For each $i$, Lemma \ref{lem3d} gives a subgroup $H_i\le N_i$ such that $X_i^{H_i}\ne0$ and
\begin{equation} \label{eq3i2}
[N_i:H_i]\le\exp(c_1(\log 2a_i)^2).
\end{equation}
Set $H_N:=H_1\times\cdots\times H_t$, where $H_N=1$ if $t=0$. Then $X^{H_N} \cong X_1^{H_1}\otimes\cdots\otimes X_t^{H_t} \ne0$, so $W^{H_N}\ne0$. Furthermore, \eqref{eq3i1} and \eqref{eq3i2} imply that
\begin{equation} \label{eq3i3}
[N:H_N]
\le\exp\left(c_1\sum_{i=1}^t(\log 2a_i)^2\right)
\le\exp\left(c_1 \left(\sum_{i=1}^t 2\log a_i\right)^2 \right)
\le\exp(4c_1(\log n)^2).
\end{equation}

Choose $0\ne w\in W^{H_N}$. By Schur's lemma \cite[Section 2.2, Proposition 4]{Ser77}, $Z$ acts on $W$ by scalars. Since $\Gamma \le \GL(W)$, this action is faithful. Hence $Z$ is cyclic, as every finite subgroup of $\bC^\times$ is cyclic. A generator of $Z$ acts by an eigenvalue of order $|Z|$, so $|Z|\le\Omega(\Gamma)$.
By Lemma \ref{lem3e}, \eqref{eq3i3} and the inequality $|Z|\le\Omega(\Gamma)$, we have
\begin{align*}
|\Gamma\cdot w|
&\le[\Gamma:H_N] \\
&=[\Gamma:ZN][ZN:N][N:H_N] \\
&\le[\Gamma:ZN]|Z|[N:H_N] \\
&\le \Omega(\Gamma) n^{2\log_2 n+5}\exp(4c_1(\log n)^2) \\
&\le \Omega(\Gamma)\exp(c_2(\log 2n)^2)
\end{align*}
for an absolute constant $c_2>0$. This completes the proof.
\end{proof}

\begin{proof}[Proof of Theorem \ref{thm1b}]
Let $V:=U^\vee$ and choose an irreducible complex constituent $W\subset V_{\bC}:=V\otimes_{\bQ}\bC$. Write $n:=\dim_{\bC}W\le d$ and let $\Gamma$ be the image of $G$ in $\GL(W)$. Choose $\gamma\in\Gamma$ whose eigenvalue $\alpha$ satisfies $\mr{ord}(\alpha)=r := \Omega(\Gamma)$, and let $g\in G$ be a preimage of $\gamma$. 
Since $W$ is a $G$-subrepresentation of $V_{\bC}$, $\alpha$ is also an eigenvalue of $g$ acting on $V_{\bC}$. As the action of $g$ on $V$ is defined over $\bQ$ and $\alpha$ has order $r$, the minimal polynomial of $\alpha$ over $\bQ$ is the $r$-th cyclotomic polynomial $\Phi_r$. 

Let $P_g(x) \in \bQ[x]$ be the characteristic polynomial of $g$ acting on $V$. Then $\deg P_g = d$ and $P_g(\alpha)=0$, so $d \ge \deg \Phi_r = \phi(r)$. A simple inequality $r \le 2 \phi(r)^2$ gives $\Omega(\Gamma) =r \le 2 \phi(r)^2 \le 2d^2$. 
By Lemma \ref{lem3f}, there is $0\ne w\in W$ such that
$$
|G\cdot w|=|\Gamma\cdot w| \le 2d^2\exp(c_2(\log 2d)^2).
$$

Let $K:=\{g\in G:g\cdot w=w\}$. Then $w \in (V_{\bC})^K$ so $(V_{\bC})^K\ne0$. Moreover, $V^K=\bigcap_{g\in K}\ker(g-\mr{id}_V)$ and similarly, $(V_{\bC})^K=\bigcap_{g\in K}\ker(g-\mr{id}_{V_{\bC}})$. Since the map $g-\mr{id}_{V_{\bC}}$ is obtained from $g-\mr{id}_V$ by extension of scalars, we have $V^K\otimes_{\bQ}\bC \cong (V_{\bC})^K\ne0$ and hence $V^K\ne0$. Now $0\ne\ell\in V^K$ satisfies
$$
|G\cdot\ell| \le[G:K] =|G\cdot w| \le 2d^2\exp(c_2(\log 2d)^2).
$$
There is a constant $C>0$ such that $2d^2\exp(c_2(\log 2d)^2) \le \exp(C(\log 2d)^2)$ for every $d \ge 1$.
\end{proof}

Finally, we prove Theorem \ref{thm1a} by combining Theorems \ref{thm1b} and \ref{thm1c}.
\begin{proof}[Proof of Theorem \ref{thm1a}]
Let $C>0$ be as in Theorem \ref{thm1b} and set $A(d):=\exp(C(\log 2d)^2)$. Fix a finite nontrivial subgroup $H$ of $\GL_n(\bZ)$ and retain the notation of Section \ref{Sub12}. If $d_i:=\dim W_i$, then $\mu_i\le A(d_i)\le A(n)$ by Theorem \ref{thm1b}, hence $\Theta(H,V)\le\frac{A(n)}{2}$ and $s(H,V)-1 \le n-1$. By Theorem \ref{thm1c},
$$
N_n^{\tor}(X;H)\ll_H X^{A(n)/2}(\log X)^{n-1} \ll_H X^{\exp(4C(\log n)^2)}
$$
for every $n \ge 2$. Now \eqref{eq1a2} completes the proof.
\end{proof}

\section*{Acknowledgments}
The author was supported by the National Research Foundation of Korea (NRF) grant funded by the Korea government (MSIT) (No. RS-2024-00334558 and No. RS-2025-02262988).

\section*{Statement on AI use}

OpenAI’s ChatGPT 5.6 Pro generated initial proofs for several central arguments in this paper, including the main representation-theoretic result and its application to counting algebraic tori over $\bQ$. 
These arguments were developed through an iterative dialogue: the model initially proposed a weaker bound than the one stated in Theorem \ref{thm1a}, which was substantially improved through repeated prompting and critical revision. 
The author formulated the project, evaluated and revised the model's suggestions, independently verified all mathematical arguments, and wrote the final manuscript. The author takes full responsibility for the correctness of the results.


\end{document}